%% file: elsarticle-app.tex
\documentclass[12pt]{article}
{\begin{Sbox}\begin{minipage}}%
{\end{minipage}\end{Sbox}\fbox{\TheSbox}}%

\usepackage[psamsfonts]{amssymb}
\usepackage{amsmath, amsthm, anysize, enumerate, color, url}
\usepackage{graphicx}
\usepackage{epstopdf}
\usepackage{epsfig}
\usepackage{subfigure}

\usepackage{algorithm}
\usepackage{algpseudocode}

\usepackage{empheq}

\usepackage{natbib}
\usepackage{threeparttable}

\newtheorem{theorem}{Theorem} 
\newtheorem{corollary}{Corollary}

\usepackage[colorlinks, breaklinks,
                   linkcolor=red,
                  anchorcolor=blue,
                  citecolor=blue
                  ]{hyperref}

\usepackage{breakurl}

\theoremstyle{definition}

\usepackage{setspace}

\begin{document}

\title{Approximating the inverse of a diagonally dominant matrix with positive elements}

\author{Ting Yan\thanks{Department of Statistics, Central China Normal University, Wuhan, 430079, China.
\texttt{Email:} tingyanty@mail.ccnu.edu.cn.}
\\
Central China Normal University
}
\date{}

\maketitle

\begin{abstract}
For an $n\times n$ diagonally dominant matrix $T=(t_{i,j})_{n\times n}$ with positive elements satisfying certain
bounding conditions, we propose to use a diagonal matrix $S=(s_{i,j})_{n\times n}$ to approximate the inverse of $T$,
where $s_{i,j}=\delta_{i,j}/t_{i,i}$ and $\delta_{i,j}$ is the Kronecker delta function.
We derive an explicitly upper bound on the
approximation error, which is in the magnitude of $O(n^{-2})$.
It shows that $S$ is a very good approximation to $T^{-1}$.


\end{abstract}

\vskip 5pt


\section{Introduction}

In this paper, we consider the approximate inverse of
an $n\times n$ diagonally dominant
matrices $T=(t_{i,j})_{n\times n}$ with positive elements satisfying certain bounding conditions, i.e.,
\begin{equation}\label{eq1}
t_{i,j}>0,~~ t_{i,i}\ge \sum\limits_{j=1,j\neq i}^{n}t_{i,j},~~~ i=1,\cdots, n.
\end{equation}
It is easy to show that $T$ must be positive definite.
We propose to use a diagonal matrix $S=(s_{i,j})_{n\times n}$ to approximate the inverse of $T$,  where
\begin{equation*}
s_{i,j}=\frac{\delta_{i,j}}{t_{i,i}},
\end{equation*}
and $\delta_{i,j}$ is the Kronecker delta function.
We obtain an explicitly upper bound on the approximation error in terms of $\max_{i,j} |(T^{-1} - S)_{ij}|$, which has the magnitude of $1/n^2$.
This shows that $S$ is a very good approximation to $T^{-1}$.

The problems on inverses of nonnegative matrices have been extensively investigated; see \cite{Berman, Loewya78, Eglestona2004}.
It has applications to solving a large system of linear equations, in which a good approximate inverse of
the coefficient matrix plays an  important role in establishing fast convergence rates of iterative algorithms
\cite{Axelsson85, Benzi, Bruaset95, Zhang09}.
Within statistics, \cite{Yan} use the approximate inverse of $T$ to obtain a fast geometric rate of convergence
of an iterative sequences for solving the estimate of parameters in the node-parameter network models with dependent structures.
Further, it is used to derive the asymptotic representation of
an estimator of the model parameter.

\section{An explicit bound on the approximation error}
For a general matrix $A=(a_{i,j})$, define the matrix maximum norm:
\[
\|A\|:=\max_{i,j} |a_{i,j}|.
\]
We measure the approximation error of using $S$ to approximate $T^{-1}$ in terms of $\| T^{-1} - S \|$.
Some notations are defined as follows:
\[
m:=\min\limits_{1\le i<j\le n}t_{i,j},~~ \Delta_i:= t_{i,i}-\sum_{j=1,j\neq i}^n t_{i,j},~~
M:=\max\{ \max\limits_{1\le i<j\le n} t_{i,j}, \max\limits_{1\le i\le n} \Delta_i\}.
\]
Note that $M\ge m >0$.
Let
\begin{equation}
\label{definition-cmM}
C(m, M) 
=  \frac{2(n-2)m}{ nM+(n-2)m} - \frac{ (n-2)Mm }{ [(n-2)m+M][ (n-2)m + 2M ] } - \frac{M}{m(n-1)} .
\end{equation}
The approximate error is formally stated below.

\begin{theorem}\label{dense-theorem}
If $C(m, M)>0$, then for $n\ge 3$, we have
\[
\| T^{-1}-S  \|\le \frac{ M}{m^2(n-1)^2C(m, M)}.
\]
\end{theorem}

\input{proof.tex}

\setlength{\itemsep}{-1.5pt}
\setlength{\bibsep}{0ex}
\bibliography{mybibfile}
\bibliographystyle{apa}

\end{document}

%% file: proof.tex
\begin{proof}
Let $I_n$ be the $n\times n$ identity matrix. Define
\[
F=T^{-1}-S,~~V=(v_{ij})=I_n-TS,~~W=(w_{ij})=SV.
\]
Then, we have the recursion:
\begin{equation}\label{recursion}
F=T^{-1}-S=(T^{-1}-S)(I_n-TS)+S(I_n-TS)=FV+W.
\end{equation}
A direct calculation gives that
\begin{equation}
      \label{vij}
v_{i,j}=\delta_{i,j}-\sum_{k=1}^n t_{i,k}s_{k,j}=\delta_{i,j}-\sum_{k=1}^n t_{i,k}\frac{\delta_{k,j}}{ t_{j,j} }
       =(\delta_{i,j}-1)\frac{t_{i,j}}{t_{j,j}},
\end{equation}
and
\begin{equation}
\label{wij}
w_{i,j}=\sum_{k=1}^n s_{i,k}v_{k,j}=\sum_{k=1}^n \frac{\delta_{i,k}}{t_{i,i}} [ (\delta_{k,j}-1)\frac{t_{k,j}}{t_{j,j}}]
      =\frac{(\delta_{i,j}-1)t_{i,j}}{t_{i,i}t_{j,j}}.
\end{equation}
Recall that $m\le t_{i,j} \le M$ and $(n-1)m \le t_{i,i} \le nM$.
When $i\neq j$, we have
\begin{equation*}
0<\frac{t_{i,j}}{t_{i,i}t_{j,j}}\le \frac{M}{m^2(n-1)^2},
\end{equation*}
such that for three different subscripts $i,j,k$,
\begin{eqnarray*}
|w_{i,i}|=  0,~~~
|w_{i,j}| \le  \frac{M}{m^2(n-1)^2}, \\
|w_{i,j}-w_{i,k}|  \le  \frac{M}{m^2(n-1)^2}, ~~~
|w_{i,i}-w_{i,k}|  \le  \frac{M}{m^2(n-1)^2}.
\end{eqnarray*}
It follows that
\begin{equation}\label{wijin}
\max(|w_{i,j}|, |w_{i,j}-w_{i,k}|)\le \frac{M}{m^2(n-1)^2}, ~~~~~ \mbox{for all $i,j,k$}.
\end{equation}

We use the recursion \eqref{recursion} to obtain a bound of the approximate error $\|F\|$.   By
\eqref{recursion} and \eqref{vij}, for any $i$, we have
\begin{equation}\label{fij}
f_{i,j}=\sum_{k=1}^n f_{i,k}[(\delta_{k,j}-1)\frac{t_{k,j}}{t_{j,j}}] +w_{i,j}, ~~~~~j=1,\cdots,n.
\end{equation}
Thus, to prove Theorem 1, it is sufficient to show that for any $i,j$,
\[
|f_{i,j}|\le  \frac{M}{m^2C(M,m)(n-1)^2}.
\]
Define $f_{i,\alpha}=\max\limits_{1\le k\le n}f_{i,k}$ and
$f_{i,\beta}=\min\limits_{1\le k\le n} f_{i,k}$.

First, we will show that $f_{i,\beta}\le 0$.
Since for any fixed $i$,
\begin{equation}
\label{hold}
\sum_{k=1}^n f_{i,k}t_{k,i} =  \sum_{k=1}^n \left( [T^{-1}]_{i,k} - \frac{\delta_{i,k}}{t_{i,i}} \right)t_{k,i}
= 1 - 1=0,
\end{equation}
we have
\[
f_{i,\beta} \sum_{k=1}^n t_{k,i} \le \sum_{k=1}^n f_{i,k}t_{k,i} = 0.
\]
 It
follows that $f_{i,\beta}\le 0$.
With similar arguments, we have that
$f_{i,\alpha}\ge 0$.

Recall that $\Delta_\alpha = t_{\alpha,\alpha} - \sum_{k=1, k\neq \alpha}^n t_{k,\alpha}$.
Since
\[
t_{\alpha,\alpha}= - \{\sum_{k=1}^n [(\delta_{k,\alpha}-1)t_{k,\alpha} -\delta_{k,\alpha} \Delta_{\alpha}]\},
\]
we have the identity
\begin{equation}
\label{eq:fibeta}
f_{i,\beta}=-\sum_{k=1}^n f_{i,\beta}\frac{(\delta_{k,\alpha}-1)t_{k,\alpha} - \delta_{k,\alpha} \Delta_{\alpha}}{t_{\alpha,\alpha}}.
\end{equation}
Similarly, we have
\begin{equation}
\label{eq:fibeta2}
f_{i,\beta}=-\sum_{k=1}^n f_{i,\beta}\frac{(\delta_{k,\beta}-1)t_{k,\beta} - \delta_{k,\beta} \Delta_{\beta}}{t_{\beta,\beta}}.
\end{equation}
By combining \eqref{fij} and \eqref{eq:fibeta}, where we set $i=\alpha$ in \eqref{fij}, it yields that
\begin{equation}\label{falpha}
f_{i,\alpha} +  f_{i,\beta}=\sum_{k=1}^n
(f_{i,k}-f_{i,\beta}) \frac{[(\delta_{k,\alpha}-1)t_{k,\alpha} -\delta_{k,\alpha} \Delta_{\alpha}]}{t_{\alpha,\alpha}}
+w_{i,\alpha}.
\end{equation}
Again, by combining \eqref{fij} and \eqref{eq:fibeta2}, we have
\begin{equation}\label{fbeta}
2f_{i,\beta} = \sum_{k=1}^n
(f_{i,k}-f_{i,\beta})\frac{[(\delta_{k,\beta}-1)t_{k,\beta} -\delta_{k,\beta} \Delta_{\beta}]}{t_{\beta,\beta}}
+w_{i,\beta}.
\end{equation}
By subtracting \eqref{fbeta} from \eqref{falpha}, we get
\begin{equation}\label{falphabeta}
\begin{array}{lll}
&&f_{i,\alpha}-f_{i,\beta}\\
&=&\sum_{k=1}^n (f_{i,k}-f_{i,\beta})[(\delta_{k,\alpha}-1)\frac{t_{k,\alpha}}{t_{\alpha,\alpha}}
-(\delta_{k,\beta}-1)\frac{t_{k,\beta}}{t_{\beta,\beta}}]\\
&& +w_{i,\alpha}-w_{i,\beta}
-(\frac{\Delta_\beta}{t_{\beta,\beta}}-\frac{\Delta_\alpha}{t_{\alpha,\alpha}})f_{i,\beta}.
\end{array}
\end{equation}
Let $\Omega=\{k:(1-\delta_{k,\beta})t_{k,\beta}/t_{\beta,\beta}
\ge (1-\delta_{k,\alpha})t_{k,\alpha}/t_{\alpha,\alpha} \}$ and define $\lambda := |\Omega|$.
Note that $1\le \lambda \le n-1$. Then,
\begin{eqnarray}\nonumber
&&\sum_{k=1}^n(f_{i,k}-f_{i,\beta})[(\delta_{k,\alpha}-1)\frac{t_{k,\alpha}}{t_{\alpha,\alpha}}
-(\delta_{k,\beta}-1)\frac{t_{k,\beta}}{t_{\beta,\beta}}] \\
\nonumber &\le &\sum_{k\in \Omega}(f_{i,k}-f_{i,\beta})[
(1-\delta_{k,\beta})\frac{t_{k,\beta}}{t_{\beta,\beta}}-(1-\delta_{k,\alpha})\frac{t_{k,\alpha}}{t_{\alpha,\alpha}}] \\
\nonumber
&\le &(f_{i,\alpha}-f_{i,\beta}) [\frac{\sum_{k\in
\Omega}t_{k,\beta}}{t_{\beta,\beta}}-\frac{\sum_{k\in
\Omega}(1-\delta_{k,\alpha})t_{k,\alpha}}{t_{\alpha,\alpha}}] \\
\nonumber
 & \le & (f_{i,\alpha}-f_{i,\beta}) \left[\frac{\lambda M
}{\lambda M+(n-1-\lambda)m}-\frac{(\lambda-1) m}{(\lambda-1)
m+(n-\lambda)M+M} \right].\\
\label{yyy}
\end{eqnarray}
We will obtain the maximum value of the expression in the above bracket through dividing it into two
functions $f(\lambda)$ and $g(\lambda)$ of $\lambda$, where
\begin{eqnarray*}
f(\lambda) & = & \frac{\lambda M}{\lambda M+ (n-1-\lambda)m }-
\frac{(\lambda-1)m}{ (\lambda-1)m+ (n-\lambda)M}, \\
g(\lambda) & = &\frac{ (\lambda-1)m }{ (\lambda-1)m + (n-\lambda)M } - \frac{ (\lambda-1)m }{ (\lambda-1)m + (n-\lambda)M + M}.
\end{eqnarray*}
We first derive the maximum value of $f(\lambda)$.
There are two cases to consider the maximum value of
$f(\lambda)$ in the range of
$\lambda\in [1, n-1]$.\\
Case I: When $M=m$, it is easy to show $f(\lambda)=1/(n-1)$.\\
Case II: $M\neq m$.
A direct calculation gives that
\begin{eqnarray*}
f^\prime(\lambda) &=& \frac{ (n-1)Mm}{[\lambda M+
(n-1-\lambda)m]^2 } -
\frac{ (n-1)Mm}{ [(\lambda-1)m+(n-\lambda)M]^2} \\
&=& \frac{ (n-1)Mm [(n-2\lambda)(M-m)] [\lambda
M+ (n-1-\lambda)m+(\lambda-1)m+ (n-\lambda)M]}{[\lambda M+
(n-1-\lambda)m]^2[(\lambda-1)m+ (n-\lambda)M]^2 }
\end{eqnarray*}
and
\begin{equation*}
f^{\prime\prime}(\lambda)=-2(M-m)Mm(n-1)\left(\frac{ 1}{[\lambda M+
(n-1-\lambda)m]^3 } + \frac{ 1}{ [(\lambda-1)m+
(n-\lambda)M]^3} \right).
\end{equation*}
Since $f^{\prime\prime}(\lambda)\le 0$ when $\lambda \in[1, n-1]$, $f(\lambda)$ is a convex function of $\lambda$ ($\in [1, n-1]$)
such that $f(\lambda)$ takes its maximum value at $\lambda=n/2$ when $1\le \lambda\le n-1$.  Note that
\begin{eqnarray*}
f(\frac{n}{2}) & = & \frac{ nM-(n-2)m}{ nM+(n-2)m}.
\end{eqnarray*}
So we have
\begin{equation}\label{flambdamax}
\sup_{\lambda \in [0,n-1] } f(\lambda) \le \frac{ nM-(n-2)m}{ nM+(n-2)m}.
\end{equation}
Next, we obtain the maximum value of $g(\lambda)$.
Since
\begin{equation*}
g^\prime(\lambda)= \frac{ Mm [M^2( (n-\lambda)^2+2(n-\lambda)(\lambda-1)+n-1 )+(2Mm-m^2)(\lambda-1)^2]}
{[(\lambda-1)m + (n-\lambda)M]^2[(\lambda-1)m + (n-\lambda)M + M]^2},
\end{equation*}
$g^\prime(\lambda)>0$ when $1\le\lambda\le n-1$.
So $g(\lambda)$ is an increasing function on $\lambda$
 such that
\begin{equation}\label{glambda-ineq}
0\le \sup_{ \lambda \in [1, n-1] } g(\lambda)\le g(n-1)= \frac{ (n-2)Mm }{ [(n-2)m+M][ (n-2)m + 2M ] }.
\end{equation}
By combining \eqref{flambdamax} and \eqref{glambda-ineq}, we have
\begin{eqnarray}
\nonumber
&&\sup_{1\le \lambda\le n-1}[\frac{\lambda M
}{\lambda M+(n-1-\lambda)m}-\frac{(\lambda-1) m}{(\lambda-1)
m+(n-\lambda)M+M}] \\
\nonumber
&\le & \sup_{1\le \lambda\le n-1}f(\lambda) +\sup_{1\le \lambda \le n-1} g(\lambda)\\
\nonumber
&\le & \frac{1}{n-1}I(M=m)+   \frac{nM-(n-2)m}{ nM+(n-2)m}I(M\neq m)+ \frac{ (n-2)Mm }{ [(n-2)m+M][ (n-2)m + 2M ] } \\
\label{upper-full}
&=&\frac{nM-(n-2)m}{ nM+(n-2)m}+\frac{ (n-2)Mm }{ [(n-2)m+M][ (n-2)m + 2M ] },
\end{eqnarray}
where $I(\cdot)$ is an indictor function.
By combining \eqref{falphabeta}, \eqref{yyy} and \eqref{upper-full}, we have
\begin{equation}
\label{eq:aaaa}
\begin{array}{rcl}
f_{i,\alpha}-f_{i,\beta}
& \le & \left\{\frac{nM-(n-2)m}{ nM+(n-2)m}+\frac{ (n-2)Mm }{ [(n-2)m+M][ (n-2)m + 2M ] }\right\}(f_{i,\alpha}-f_{i,\beta})
\\
&& + |w_{i,\alpha}-w_{i,\beta}|
+ \left|\frac{\Delta_\beta}{t_{\beta,\beta}}-\frac{\Delta_\alpha}{t_{\alpha,\alpha}}\right||f_{i,\beta}|.
\end{array}
\end{equation}
Since $f_{i,\alpha}\ge |f_{i,\beta}|$ and $f_{i,\beta}\le 0$, we have
\begin{equation} \label{eqn:betaalpha}
\left|\frac{\Delta_\beta}{t_{\beta,\beta}}-\frac{\Delta_\alpha}{t_{\alpha,\alpha}}\right||f_{i,\beta}|
\le \left|\frac{\Delta_{\beta}}{t_{\beta\beta}} - \frac{\Delta_{\alpha}}{t_{\alpha\alpha}} \right|(f_{i\alpha} - f_{i\beta}) \le \frac{M}{m(n-1)} (f_{i\alpha} - f_{i\beta}).
\end{equation}
Recall the definition of $C(m, M)$ in \eqref{definition-cmM}. By combining \eqref{eq:aaaa} and \eqref{eqn:betaalpha}, it yields
\[
(f_{i,\alpha}-f_{i,\beta})C(m, M) \le |w_{i,\alpha}-w_{i,\beta}| \le \frac{ M}{m^2(n-1)^2 }.
\]
Consequently,
\begin{eqnarray*}
\max_{j=1,\cdots, n}|f_{i,j}| \le  f_{i,\alpha}-f_{i,\beta}  \le  \frac{M}{m^2 (n-1)^2 C(M, m)}.
\end{eqnarray*}
This completes the proof.
\end{proof}

We discuss the condition $C(m,M)>0$. $C(m,M)$ can be represented as
\[
C(m, M) = \frac{ 2(n-2)m }{ nM + (n-2)m } - \frac{ (n-2) (M/m) }{ [ (n-2) + M/m ][ (n-2) + 2 M/m ]} - \frac{ M/m }{n-1}.
\]
So if $M/m=o(n)$, then for large $n$
\[
C(m, M) =  \frac{ 2m }{ M + m}  + o(1).
\]
Then we immediately have the corollary.

\begin{corollary}
If $M/m=o(n)$, then for large $n$,
\[
\| T^{-1}-S  \| = O\left( \frac{M^2}{m^3n^2 } \right).
\]
\end{corollary}

\section{Discussion}
The bound on the approximation
error  in Theorem 1  depends on $m$, $M$ and $n$. When $m$ and $M$
are bounded by a constant, all the elements of $T^{-1}-S$ are of
order $O(1/n^2)$ as $n\to\infty$, uniformly.
Therefore we conjecture that $T$ may belong to inverse $M$-matrices.
The interested readers can refer to \cite{Berman, Minc1988}.

We illustrate by an example that the bound on the
approximation error in Theorem 1 is optimal in the sense that any bound in the
form of $K(m,M)/f(n)$ requires $f(n)=O(n^2)$ as $n\to\infty$.
Assume that the matrix $T$ consists of the elements:
$t_{i,i}=(n-1)M, i=1,\cdots,n-1; t_{n,n}=(n-1)m$ and
$t_{i,j}=m, i,j=1,\cdots, n; i\neq j$, which satisfies \eqref{eq1}. By the Sherman-Morrison formula, we
have
\begin{eqnarray*}
(T^{-1})_{i,j}&=&\frac{\delta_{i,j} }{(n-1)M-m}-\frac{m}{[(n-1)M-m]^2}, i,j=1,\cdots,n-1\\
(T^{-1})_{n,j}&=&\frac{\delta_{n,j} }{(n-2)m}-\frac{1}{(n-2)[(n-1)M-m]},~~j=1,\cdots,n.
\end{eqnarray*}
In this case, the elements of $S$ are
\begin{eqnarray*}
S_{i,j} & = &\frac{\delta_{i,j}}{(n-1)M}-\frac{1}{n(n-1)m}, ~~~i,j=1,\cdots, n-1;i\neq j,\\
S_{n,j} & = &\frac{\delta_{n,j}}{(n-1)m}-\frac{1}{n(n-1)m}, ~~~j=1,\cdots, n.
\end{eqnarray*}
It is easy to show that the bound of $||T^{-1}-S||$ is $O(
\frac{1}{(n-1)^2m} )$. This suggests that the rate $1/(n-1)^2$ is
optimal. On the other hand, there is a gap between $1/m$ and
$O(M^2/m^3)$ which implies that there might be space for
improvement. It is interesting to see if the bounds in Theorem 1
can be further relaxed.